\documentclass[11pt]{amsart}

\usepackage[pdftex]{graphicx}%
\usepackage[dvipsnames]{xcolor}

\usepackage{a4wide}

\usepackage{amsfonts}
\usepackage{amsmath}
\usepackage{amssymb}
\usepackage{amsthm}
\usepackage{mathtools}
\usepackage{tikz}
\usetikzlibrary{patterns}

\usepackage{paralist}
\usepackage[colorlinks,cite color=blue,pagebackref=true,pdftex]{hyperref}

\usepackage[T1]{fontenc}

\usepackage{caption}
\usepackage{subcaption}

\newcommand\Defn[1]{\textbf{\color{black}#1}}
\newcommand\Def[1]{\Defn{#1}}

\renewcommand\phi{\varphi}
\newcommand\eps{\varepsilon}
\renewcommand\emptyset{\varnothing}
\newcommand\R{\mathbb{R}}

\DeclareFontFamily{U}{mathb}{\hyphenchar\font45}
\DeclareFontShape{U}{mathb}{m}{n}{
      <5> <6> <7> <8> <9> <10> gen * mathb
      <10.95> mathb10 <12> <14.4> <17.28> <20.74> <24.88> mathb12
      }{}
\DeclareSymbolFont{mathb}{U}{mathb}{m}{n}

\DeclareMathSymbol{\precneq}{3}{mathb}{"AC}

\newcommand\Z{\mathbb{Z}}
\newcommand\Rnn{\R_{\ge0}}
\newcommand\inner[1]{\langle {#1} \rangle}
\newcommand\defeq{\coloneqq}

\newcommand\op{^\mathrm{op}}
\DeclareMathOperator{\relint}{relint}

\DeclareMathOperator{\lin}{lin}
\DeclareMathOperator{\lineal}{lineal}

\DeclareMathOperator{\cone}{cone}
\DeclareMathOperator{\vol}{vol}

\newtheorem{thm}{Theorem}[section]%
\newtheorem{cor}[thm]{Corollary}%
\newtheorem{lem}[thm]{Lemma}%
\newtheorem{prop}[thm]{Proposition}

\theoremstyle{definition}

\newtheorem{rem}[thm]{Remark}

\title{Fan Valuations and spherical intrinsic volumes}

\author{Spencer Backman}

\address{Department of Mathematics and Statistics,
University of Vermont, Burlington,
Vermont, USA.}
\email{Spencer.Backman@uvm.edu}

\author{Sebastian Manecke} 
\author{Raman Sanyal}

\address{Institut f\"ur Mathematik, Goethe-Universit\"at Frankfurt, Germany} 
\email{manecke@math.uni-frankfurt.de}
\email{sanyal@math.uni-frankfurt.de}



\date{\today}

\begin{document} 
\begin{abstract}
    We generalize valuations on polyhedral cones to valuations on fans.  For
    fans induced by hyperplane arrangements, we show a correspondence between
    rotation-invariant valuations and deletion-restriction invariants.  In
    particular, we define a characteristic polynomial for fans in terms of
    spherical intrinsic volumes and show that it coincides with the usual
    characteristic polynomial in the case of hyperplane arrangements. This
    gives a simple deletion-restriction proof of a result of Klivans--Swartz.

    The metric projection of a cone is a piecewise-linear map, whose
    underlying fan prompts a generalization of spherical intrinsic volumes to
    indicator functions. We show that these \emph{intrinsic indicators} yield
    valuations that separate polyhedral cones. Applied to hyperplane
    arrangements, this generalizes a result of Kabluchko on projection
    volumes.
\end{abstract}
\maketitle

\newcommand\Faces{\mathcal{F}}%

\newcommand\Fans{\mathbf{Fans}}%
\newcommand\Cones{\mathbf{Cones}}%
\newcommand\Fan{\mathcal{N}}%
\newcommand\Delete{\backslash}%
\newcommand\Restrict{/}%
\newcommand\Arr{\mathcal{A}}%
\newcommand{\Flats}{\mathcal{L}}%
\newcommand{\Bot}{\hat{0}}%
\newcommand{\Top}{\hat{1}}%
\newcommand\ochi{\overline{\chi}}%
\newcommand\Moreau{\mathcal{M}}%

\section{Introduction}\label{sec:intro}%
Let $\mathcal{S}$ be a collection of sets closed under taking intersections. A
map $\phi$ from $\mathcal{S}$ into some abelian group $G$ is a \Def{valuation} if
\[
    \phi(S \cup T) \ = \ \phi(S) + \phi(T) - \phi(S \cap T) \, ,
\]
for any $S,T \in \mathcal{S}$ such that $S \cup T \in \mathcal{S}$. For 
geometric objects such as convex
polytopes, polyhedra, or subspaces, valuations are a gateway between geometry
and combinatorics, amply demonstrated in~\cite{barvinok-intpts,RotaKlain}.
In particular Ehrenborg and Readdy~\cite{EhrenborgReaddy} showed how
generalizations of Zaslavsky's famous formula for the number of regions of a
hyperplane arrangement can be easily inferred using valuations.  The purpose of
this note is to further the ties between geometry and combinatorics
by studying valuations on more general arrangements of geometric objects.

A \Def{fan} in $\R^d$ is a finite collection $\Fan$ of equi-dimensional
polyhedral cones that pairwise meet in faces. We further require that all
cones have the same linear span and thus speak of fans relative to subspaces
of $\R^d$.  Although the collection of fans $\Fans_d$ is a meet-semilattice
with respect to common refinement to which valuations can be
generalized~\cite{gromer,rota-valuation}, there is no natural join operation.
To define valuations on fans, we adapt Sallee's notion of \emph{weak}
valuations~\cite{Sallee} on polytopes: A map $\phi : \Fans_d \to G$ is a
\Def{fan valuation} if for every fan $\Fan \in \Fans_d$ and linear hyperplane
$H$
\[
    \phi(\Fan) \ = \ \phi(\Fan \cap H^{\le})  + \phi(\Fan \cap H^{\ge}) -
    \phi(\Fan \cap H) \, ,
\]
where $H^\le,H^\ge$ denote the two halfspaces induced by $H$. We defer
precise definitions to Section~\ref{sec:fan_vals}. Any arrangement
$\Arr$ of linear hyperplanes in some subspace $U \subseteq \R^d$
induces a fan $\Fan(\Arr)$. In Proposition~\ref{prop:fan_del_contr}, we show
that fan valuations on arrangements satisfy 
\[
        \phi(\Fan(\Arr)) \ = \ \phi(\Fan(\Arr \Delete H)) + \phi(\Fan(\Arr
        \Restrict
        H)) \, .
\]
Such invariants are closely related to deletion-contraction invariants on
(simple, loopless) matroids (\cite[Sect.~3.11]{EC1}, \cite{GZ},
\cite{BrylawskiOxely}). The main difference is that we do not impose special
treatment when $H$ is a coloop of $\Arr$. The benefit is that these
\emph{weak} deletion-restriction invariants have the structure of an abelian
group. More precisely, let $\Flats(\Arr)$ be the lattice of flats of $\Arr$
and write $w_0(\Arr),\dots,w_d(\Arr)$ for the Whitney numbers of the first
kind. Then the group weak deletion-restriction invariants is spanned by the
Whitney numbers.

We show that fan valuations invariant under rotation yield precisely the weak
deletion-restriction invariants.
\begin{thm}\label{thm:fan_val_main}
    Let $\phi : \Fans_d \to G$ be a rotation-invariant fan
    valuation. Then there are $g_0,\dots,g_d \in G$ such that for any
    hyperplane arrangement $\Arr$
    \[
        \phi(\Fan(\Arr)) \ = \ 
        g_0 w_0(\Arr) + 
        \cdots + 
        g_d w_d(\Arr) \, .
    \]
    Conversely, for any deletion-restriction invariant $\psi$ on hyperplane
    arrangements, there is a fan valuation $\phi$ with $\psi(\Arr) =
    \phi(\Fan(\Arr))$.
\end{thm}

For the latter part, we consider fan valuations induced by \Def{spherical
intrinsic volumes}, also known as projection volumes. The $k$-th spherical
intrinsic volume $v_k(C)$ of a cone $C \subset \R^d$ is the probability that
the point of $C$ closest to $x \in B_{d}$ is contained in the relative
interior of a face $F_x \subseteq C$ of dimension $k$. The definition is
extended to fans by setting $v_k(\Fan) \defeq \sum_{C \in \Fan} v_k(C)$. With
that, we define the (unsigned) characteristic polynomial of a fan by
\[
    \ochi_{\Fan}(t) \ \defeq \ v_0(\Fan) + v_1(\Fan) t +  \cdots +
    v_d(\Fan) t^d  \, .
\]
We argue that $\ochi_{\Fan}(t)$ is a suitable generalization of the
characteristic polynomial of a hyperplane arrangement. We show that it
satisfies Zaslavsky's fundamental results: 
\[
    \ochi_{\Fan}(1) \ = \ \# \text{ regions of $\Fan$} 
    \qquad \text{ and } \qquad 
    \ochi_{\Fan}(-1) \ = \ 0  \, .
\]
We derive the latter from an identity of
Hug--Kabluchko~\cite{hug2016inclusion}, for which we provide a self-contained
proof (Theorem~\ref{thm:F-N_FC}). 

Most importantly, we show $\ochi_{\Fan(\Arr}(t) = (-1)^{\dim
\Arr}\chi_{\Arr}(-t)$, where $\chi_{\Arr}(t)$ is the usual characteristic
polynomial of an arrangement. This gives a simple deletion-restriction proof
of the main result of Klivans--Swartz~\cite{KlivansSwartz} that identifies
projection volumes with the Whitney numbers of $\Flats(\Arr)$.

Whereas Hadwiger's famous classification theorem~\cite{Hadwiger}
states that the linear space of continuous and rigid-motion valuations
on convex bodies is spanned by the usual intrinsic volumes, there is
no such result for spherical convex sets. McMullen~\cite[Problem
49]{ProbInConv} conjectured that the linear space
of continuous and rotation-invariant valuations on spherical convex
bodies is spanned by the spherical intrinsic volumes. From this
perspective, Theorem~\ref{thm:fan_val_main} together with
the result of Klivans--Swartz can be seen as an indication for this
conjecture.

In the second part of the paper, we take a more refined look at spherical
intrinsic volumes. The collection of points $x \in \R^d$ such that the nearest
point in $C$ is contained in a fixed face $F \subseteq C$ is a polyhedral cone
$\Pi_F(C)$ and $\Moreau(C) \defeq \{ \Pi_F(C) : F \subseteq C \text{ face}\}$
is a complete fan, which we call the \Def{Moreau fan} of $C$. The face lattice of
$\Moreau(C)$ is given by the interval poset of the face lattice of $C$. Such
fan structures were considered by Bj\"orner under the name of
\emph{anti-prisms} in connection with a question of Lindstr\"om and our
findings reconfirm results announced in~\cite{bjorner}.

It is known that $C \mapsto v_k(C)$ is a cone valuation~\cite{McM-angle}.
We prove a generalization that this holds on the level of simple indicator
functions: Consider the set $\Pi_k(C) = \bigcup_F \Pi_F(C)$, where the union
is over all $k$-dimensional faces of $C$. Its \emph{simple} indicator is the
function $V_k(C) : \R^d \to \Z$ that disagrees with the indicator of
$\Pi_k(C)$ only on a set of measure zero. We show that $C \mapsto V_k(C)$ is a
valuation (Theorem~\ref{thm:Vk_val}) and that $C$ can be recovered from $V_k(C)$ for
all $\dim \lineal(C) \le k \le \dim C$ with $2k \neq d$ (Theorem~\ref{thm:oV-injec}).

For the function $V_k(\Fan(\Arr)) = \sum_{C \in \Fan(\Arr)}V_k(C)$ of a
hyperplane arrangement $\Arr$ it follows that $V_k(\Fan(\Arr))(x) = w_k(\Arr)$
for all generic $x \in \R^d$.  Kabluchko~\cite{Kabluchko2020}
showed that the exceptional set $\{x : V_k(\Fan(\Arr))(x) \neq w_k(\Arr) \}$
coincides with the support of a hyperplane arrangement. We generalize
Kabluchko's result in Theorem~\ref{thm:Vk_arr} with a short proof that also
allows us to give a simple interpretation for the associated arrangement. 

\textbf{Acknowledgements.} Work on this project started at the Mathematical
Sciences Research Institute in Berkeley, California, during the Fall 2017
semester on \emph{Geometric and Topological Combinatorics}. We acknowledge
support by National Science Foundation under Grant
No.~DMS-1440140 during our stay at the MSRI as well as from the
DFG-Collabora\-tive Research Center, TRR 109 ``Discretization in Geometry and
Dynamics''.

\section{Fan valuations and deletion-restriction invariants}\label{sec:fan_vals}

Throughout, we assume that polyhedral cones are nonempty and contain the
origin. For a polyhedral cone $C \subseteq \R^d$, the \Def{linear hull}
$\lin(C)$ is the inclusion-minimal linear subspace containing $C$ and the
\Def{lineality space} is the inclusion-maximal linear subspace contained in $C$.
A \Def{fan} in $\R^d$ is a finite collection $\Fan$ of polyhedral cones such
that for all $C, C' \in \Fan$
\begin{compactenum}[\rm (i)]
    \item $C \cap C'$ is a face of $C$ and
    \item $\lin(C) = \lin(C')$.
\end{compactenum}
It follows that all cones are of the same dimension. We set $\dim \Fan \defeq
\dim C$ and $\lin(\Fan) \defeq  \lin(C)$ as well as $\lineal(\Fan) \defeq
\lineal(C)$ for any $C \in \Fan$. The \Def{rank} of $\Fan$ is $r(\Fan) =
\dim(\Fan) - \dim \lineal(\Fan)$.  Denote by $\Fans_d$ the collection of fans
in $\R^d$. We define for any set $S \subseteq \lin(\Fan)$
\[
    \Fan \cap S \defeq \ \{ C \cap S : C \in \Fan,\; \relint(C) \cap S \neq
    \emptyset \} \, .
\]

A map $\phi$ from $\Fans_d$ into some abelian group is a \Def{fan valuation}
if $\phi(\emptyset) = 0$ and for any $\Fan \in \Fans_d$ and hyperplane $H
\subseteq \lin(\Fan)$ 
\[
    \phi(\Fan) \ = \ \phi(\Fan \cap H^{\le})  + \phi(\Fan \cap H^{\ge}) -
    \phi(\Fan \cap H) \, ,
\]
where $H^\le, H^\ge$ denote the two closed halfspaces induced by $H$.

Let $\Cones_d$ be the intersectional family of polyhedral cones in $\R^d$.
Every cone $C \in \Cones_d$ gives rise to a fan $\{C\} \in \Fans_d$ and a fan
valuation gives rise to map $\phi'$ on $\Cones_d$ satisfying
\[
    \phi'(C)  \ = \ \phi'(C \cap H^\le) +
    \phi'(C \cap H^\ge) - \phi'(C \cap H) \, .
\]
Such a map is called a \Def{weak valuation}~\cite{Sallee}. Every valuation on
cones is naturally a weak valuations and Sallee's~\cite{Sallee} arguments imply
that every weak valuation on polyhedral cones is a valuation. In particular,
if $\phi'$ is a cone valuation, then
\begin{equation}\label{eqn:induced_val}
    \phi(\Fan) \ \defeq \ \sum_{C \in \Fan} \phi'(C)
\end{equation}
is a fan valuation. The next result shows that every valuation is of that
form.
\begin{prop}\label{prop:cone_val}
    Let $\phi$ be a fan valuation. Then 
    \[
        \phi(\Fan) \ = \ \sum_{C \in \Fan} \phi(\{C\}) \, .
    \]
\end{prop}

Thus, we will write $\phi(C)$ instead of $\phi(\{C\})$ from now on.

\begin{proof}
    The claim follows trivially if $\Fan$ is empty or if it consists
    of a single cone. If $\dim \Fan = 1$, then the only nontrivial
    case is $\Fan = \{\Rnn c, -\Rnn c\}$ for some
    $c \in \R^d \setminus\{0\}$. Now $H = \{0\}$ is the unique linear
    hyperplane in $\lin(\Fan)$ and $\Fan \cap H = \emptyset$. The
    assertion follows from the definition of fan valuations.

    Assume now that the statement holds for all fans consisting of
    less than $k$ cones and whose dimension is smaller than $e$ for some
    $k \geq 2$ and $e \ge 2$. Let $\Fan$ be a fan with $k$ cones and
    $\dim \Fan = e$. For $C,C' \in \Fan$ let $H$ be a separating
    hyperplane. Then $\Fan \cap H^\ge$ and $\Fan \cap H^\le$ have less
    than $k$ cones and $\dim(\Fan \cap H) < e$.  Let
    $X = \{C \in \Fan : \relint C \cap H \neq \emptyset\}$. Then
    \begin{align*}
      \phi(\Fan) \ &= \ \phi(\Fan \cap H^\leq) + \phi(\Fan \cap H^\geq) - \phi(\Fan \cap H)\\
      \ &= \ \sum_{C \in \Fan \setminus X} \phi(\{C\}) + \sum_{C \in X} (\phi(\{C \cap H^\leq\}) + \phi(\{C \cap H^\geq\}) - \phi(\{C \cap H\}))\\
      \ &= \ \sum_{C \in \Fan \setminus X} \phi(\{C\}) + \sum_{C \in X} \phi(\{C\}) \ = \ \sum_{C \in \Fan} \phi(\{C\})\,.\qedhere
    \end{align*}
\end{proof}

\subsection{Hyperplane arrangements}
Let $\Arr$ be a finite collection of linear hyperplanes in a subspace $L$ of
$\R^d$. The complement $L \setminus \bigcup \Arr$ is a collection of open
cones whose closures define a fan that we denote by $\Fan(\Arr)$. For a
hyperplane $H \not\in \Arr$ observe that $\Fan(\Arr) \cap H = \Fan(\Arr
\Restrict H)$, where $\Arr \Restrict H \defeq \{ H' \cap H : H' \in \Arr, H'
\neq H \}$ is the \Def{restriction} of $\Arr$ to $H$. If $H \in \Arr$, we
write $\Arr \Delete H = \{H' \in \Arr : H' \neq H\}$ for the \Def{deletion} of
$H$. A \Defn{$k$-singleton}, or simply, a singleton, is a hyperplane
arrangement consisting of a single hyperplane in some $k$-dimensional subspace
of $\R^d$, where $1 \leq k \leq d$.

\begin{prop}\label{prop:fan_del_contr}
    Let $\phi$ be a fan valuation and $\Arr$ a hyperplane arrangement
    which is not a singleton. For any $H \in \Arr$
    \[
        \phi(\Fan(\Arr)) \ = \ \phi(\Fan(\Arr \Delete H)) + \phi(\Fan(\Arr
        \Restrict
        H)) \, .
    \]
\end{prop}
\begin{proof}
    Let $H \in \Arr$ and set $\Arr' = \Arr \Delete H$.
    The valuation property yields
    \[
        \phi(\Fan(\Arr')) \ = \ \phi(\Fan(\Arr') \cap H^\le) +
        \phi(\Fan(\Arr') \cap H^\ge) -
        \phi(\Fan(\Arr') \cap H)
    \]
    We infer from Proposition~\ref{prop:cone_val} that $\phi(\Fan(\Arr)) =
    \phi(\Fan(\Arr') \cap H^\le) + \phi(\Fan(\Arr') \cap H^\ge)$. By
    definition $\phi(\Fan(\Arr' \cap H)) = \phi(\Fan(\Arr \Restrict H))$,
    which now yields the claim.
\end{proof}

\newcommand\SO{\mathrm{SO}}%
The group $\SO(\R^d)$ of rotations acts on $\Fans_d$ by $g \cdot \Fan \defeq
\{gC : C \in \Fan\}$. We call a fan valuation $\phi$ \Def{invariant} if
$\phi(g \cdot \Fan) = \phi(\Fan)$ for all $g \in \SO(\R^d)$ and $\Fan \in
\Fans_d$. Clearly, if $\phi$ is invariant, it assigns the same value to all
$k$-singletons.

We define the \Def{unsigned characteristic polynomial} $\ochi_\Arr(t)$ of an
arrangement $\Arr$ recursively as follows. If $\Arr$ is a $k$-singleton, then
$\ochi_{\Arr}(t) \defeq t^k + t^{k-1}$. If $\Arr$ consists of more than one
hyperplane, then for $H \in \Arr$
\begin{equation}\label{eqn:ochi}
    \ochi_{\Arr}(t) \ \defeq \ \ochi_{\Arr \Delete H}(t) + \ochi_{\Arr \Restrict H}(t)
    \, .
\end{equation}
The unsigned characteristic polynomial is related to the usual characteristic
polynomial (cf.~\cite[Sect.~3.11]{EC1}) by $\ochi_\Arr(t) = (-1)^{\dim \Arr}
\chi_\Arr(-t)$. The coefficients of $\ochi_\Arr(t)$ are the (unsigned)
\emph{Whitney numbers of the first kind} denoted by $w_i(\Arr)$. 

\begin{thm}\label{thm:whitney_decomp}
    Let $\phi$ be an invariant fan valuation taking values in an
    abelian group $G$. Then there are $a_0,\dots,a_d \in G$, such that
    for every arrangement $\Arr$
    \begin{equation}\label{eq:whitney_decomp}
        \phi(\Fan(\Arr)) \ = \ a_0 w_0(\Arr) + \cdots + a_{d-1} w_{d-1}(\Arr) \, .
    \end{equation}
    Moreover, the $a_i$'s are determined by the values $\phi(\Arr^k)$, where
    $\Arr^k$ is a $k$-singleton, $1 \leq k \leq d$.
\end{thm}
\begin{proof}
    Let $b_k \defeq \phi(\Fan(\Arr^k))$ for any $k$-singleton $\Arr^k$
    and set $a_{k-1} \defeq \sum_{i = k}^d (-1)^{k-i} b_i$ for
    ${1 \leq k \leq d}$. We proceed by induction. Since
    $a_{k-1} + a_k = b_k$, clearly \eqref{eq:whitney_decomp} holds
    for $k$-singletons. Otherwise, let $H \in \Arr$. Then
    \begin{align*}
      \phi(\Fan(\Arr))
      \ &= \ \phi(\Fan(\Arr \Delete H)) + \phi(\Fan(\Arr \Restrict H))\\
      \ &= \ \sum_{i = 0}^{d-1} a_i w_i(\Arr \Delete H)  + \sum_{i = 0}^{d-1}
        a_i w_i(\Arr \Restrict H)
      \ = \ \sum_{i = 0}^{d-1} a_i w_i(\Arr)\qedhere
    \end{align*}
\end{proof}

\newcommand{\sphvol}{\vartheta}%
\newcommand\MD{\mathcal{M}}%

\subsection{Spherical intrinsic volumes}
Every weak deletion-restriction invariant arises from an invariant fan
valuation, more precisely, from a combination of spherical intrinsic volumes.
To show this, we introduce a characteristic polynomial for fans.

Given a polyhedral cone $C \subseteq \R^d$ and a point $x \in \R^d$, there is
a unique point $\pi_C(x) \in C$ minimizing the Euclidean distance $\|x -
\pi_C(x)\|_2$. The map $\pi_C : \R^d \to C$ is called the \Def{metric
projection} or \Def{nearest-point map} of $C$;
cf.~\cite[Sect.~1.2]{Schneider}. Let us denote by $F_x $ the unique face of
$C$ that contains $\pi_C(x)$ in its relative interior and $\Pi_k(C) \defeq \{ x \in
\R^d : \dim F_x = k \}$.  The \Def{$k$-th spherical intrinsic volume} is given
by
\[
    v_k(C) \ \defeq \ \frac{\vol(\Pi_k(C) \cap B_d)}{\vol(B_d)} \, .
\]
In the next section we will consider the sets $\Pi_k(C)$ more closely and
in particular deduce the known fact that $C \mapsto v_k(C)$ is a
cone valuation (Corollary~\ref{cor:vk_val}). It is apparent that
$v_k$ is $\SO(\R^d)$-invariant. We write $v_k(\Fan)$ for the induced $k$-th
spherical intrinsic volume of a fan $\Fan$ and we define its
its \Def{characteristic polynomial}
\[
    \ochi_{\Fan}(t) \ \defeq \ v_0(\Fan) + v_1(\Fan) t +  \cdots +
    v_d(\Fan) t^d  \, .
\]
The characteristic polynomial shares a number of similarities with
that of a hyperplane arrangement. Since the $\Pi_k(C)$ cover $\R^d$ and
$\Pi_k(C) \cap \Pi_l(C)$ has measure zero, it follows that
\[
    \ochi_{\Fan}(1) \ = \ |\Fan| \, 
\]
and it follows from Theorem~\ref{thm:F-N_FC} in the next
section that
\[
    \ochi_{\Fan}(-1) \ = \ 0 \, ,
\]
which is a counterpart to Zaslavsky's famous results concerning
characteristic polynomials of hyperplane arrangements. %

For the case $\Fan = \Fan(\Arr)$, the spherical intrinsic volumes $v_k(\Fan) =
\sum_{C \in \Fan} v_k(C)$ were studied by Klivans
and Swartz and the following is the main result of~\cite{KlivansSwartz}.

\begin{cor}\label{cor:klivans_swartz}
    Let $\Arr$ be a linear hyperplane arrangement. Then
    \[
        \ochi_{\Fan(\Arr)}(t) \ = \ \ochi_{\Arr}(t) \, .
    \]
\end{cor}

\begin{proof}
    In light of Proposition~\ref{prop:fan_del_contr}
    and~\eqref{eqn:ochi}, it suffices to compute
    $\ochi_{\Fan(\Arr^k)}(t)$, where $\Arr^k$ is a $k$-singleton
    in some linear subspace $L$. More specifically, if
    $H^\le \subset L$ is a halfspace, then
    $v_{k}(H^\le) = v_{k-1}(H^\le) = \frac{1}{2}$ and
    $v_{j}(H^\le) = 0$ for all other $j$.  This shows
    $\ochi_{\Fan(\Arr^k)}(t) = t^k + t^{k-1}$. Hence
    $\ochi_{\Fan(\Arr)}(t)$ and $\ochi_{\Arr}(t)$ satisfy the same
    deletion-restriction recurrence with identical starting
    conditions.
\end{proof}

\section{Anti-prism fans and intrinsic volumes}\label{sec:indicator}

In this section, we take a closer look at the geometric combinatorics of
spherical intrinsic volumes by way of associated indicator functions.

\subsection{Moreau fans and anti-prisms}
Let $C \subseteq \R^d$ be a polyhedral cone. The \Def{metric projection}
$\pi_C : \R^d \to C$ associates to any point $x \in \R^d$ the unique point
$\pi_C(x) \in C$ with $\|x - \pi_C(x)\|$ minimal.  It is straightforward to
verify that if $x,y \in \R^d$ such that $\pi_C(x),\pi_C(y) \in F$ for some
face $F \subseteq C$, then $\pi_C(x+y) \in F$. Hence $\Pi_F(C) \defeq
\pi^{-1}_C(F)$ is a closed, full-dimensional polyhedral cone.
Moreau~\cite{Moreau} considered the decomposition of space into the collection
of cones
\[
    \Moreau(C) \ \defeq \ \{ \Pi_F(C) : \text{$F$ is a nonempty face of $C$} \} 
\]
and we call $\Moreau(C)$ the \Def{Moreau fan} of $C$.

\newcommand\poset{\mathcal{L}}%
The combinatorics of Moreau fans can be nicely described in terms of
Lindstr\"om's \emph{interval posets}~\cite{lindstroem}. Let $(\poset,\preceq)$
be a partially ordered set. The \Def{interval poset} $I(\poset)$ is the
collection of nonempty intervals $[a,c] = \{ b : a \preceq b \preceq c \}$
ordered by \emph{reverse} inclusion. The maximal elements are precisely
$[a,a]$ and if $\poset$ has a top and bottom element $\hat{1}$ and $\hat{0}$,
respectively, then $[\hat{0},\hat{1}]$ is the unique minimum of $I(\poset)$.
Lindstr\"om~\cite{lindstroem} asked if $I(\Faces(P))$ is the face lattice of a
polytope whenever $\Faces(P)$ is the face lattice of a polytope $P$.
Bj\"orner affirmatively answered Lindstr\"om's question for $3$-dimensional
polytopes. The complete question was resolved in the negative by
Dobbins~\cite{dobbins}.

Bj\"orner~\cite{bjorner} also announced that $I(\Faces(P))$ is the face poset of
a complete fan (or star-convex sphere), the \emph{anti-prism} fan of $P$. We
briefly reconfirm this result by showing that the Moreau fan of $C = \cone( P
\times \{1\} )$ realizes $I(\Faces(P))$. 
Let $\poset\op$ be the dual (or opposite) poset of $\poset$. Then 
\[
    I(\poset) \ \cong \ \{ (a,b)  : a \preceq
    b \} \  \subseteq \ \poset \times \poset\op \, .
\]
The \Def{face lattice} $\Faces(C)$ of $C$ is the collection of nonempty faces
of $C$ partially ordered by inclusion. This is a graded poset ranked by
dimension and we denote by $\Faces_k(C)$ the $k$-dimensional faces of $C$.
For $F \in \Faces(C)$ let
\[
    N_FC \ \defeq \ \{ c \in \R^d : \inner{c,x} \le \inner{c,y} \text{ for
    all } x \in C, y
    \in F \}
\]
be the \Def{normal cone} of $F$ at $C$. In particular the polar to $C$ is
$C^\vee = N_{\lineal(C)} C$ and hence $F \mapsto N_FC$ is an isomorphism from
$\Faces(C)$ to $\Faces(C^\vee) = \Faces(C)\op$.

Let $x \in \R^d$. It follows from $\|x - \pi_C(x)\| \le \|x - z\|$ for all $z
\in C$ such that $x-\pi_C(x) \in N_{F_x}C$. Hence $\Pi_F(C) = F + N_FC$ for all
nonempty faces $F \subseteq C$. Moreover 
\begin{equation}\label{eqn:PiF_intersect}
    (F + N_FC) \cap (G + N_GC) \ = \  (F \cap G) + (N_GC \cap N_FC) \, ,
\end{equation}
which shows the following.

\begin{prop}\label{prop:int_poset_moreau}
    Let $C \subseteq \R^d$ be a polyhedral cone. The face lattice of the
    Moreau fan $\Moreau(C)$ is isomorphic to the interval poset
    $I(\Faces(C))$.
\end{prop}

\subsection{Conical functions}
\newcommand\coneval{\mathfrak{C}}%
\newcommand{\sconeval}{\mathfrak{S}}%
For a subset $S \subseteq \R^d$, we denote its \Defn{indicator function} by
$[S] : \R^d \to \{0, 1\}$, which is defined by $[S](x) = 1$ if and only if $x
\in S$. Let $\coneval_d \subseteq \mathrm{Fun}(\R^d,\Z)$ be the abelian
subgroup spanned by $[C]$ for $C \in \Cones_d$. Since $[C] = [C \cap H^\le] +
[C \cap H^\ge] - [C \cap H]$, the map $C \mapsto [C]$ is a valuation.
Moreover, any homomorphism $\widetilde\phi : \coneval_d \to G$ gives a
valuation on cones by $ \phi(C) \defeq \widetilde\phi([C]) \, $ and
Groemer~\cite{gromer} showed every cone valuation arises that way.

\newcommand\eulerc{\epsilon}%
For a cone $C \subseteq \R^d$ define
\[
    V_k(C) \ \defeq \ \sum_{F \in \Faces_k(C)} [\Pi_F(C)] \ \in \ \coneval_d \, .
\]
This is an indicator generalization of the spherical intrinsic volumes and
$v_k(C)$ is recovered from $V_k$ by taking the spherical volume.  Note that
$V_k : \Cones_d \to \coneval_d$ is not a valuation:  For the augmentation
$\eulerc : \coneval_d \to \Z$ with $\eulerc([C]) = 1$ for all nonempty cones
$C$, we see that $\eulerc(V_k(C))$ is the number of $k$-dimensional faces of
$C$, which is not a valuation.  Nonetheless, we can view $V_k$ as a valuation
taking values in the group of \Defn{simple} indicator functions
\[
    \sconeval_d \ \defeq \ \coneval_d / \langle [C] : C \in \Cones_d, \dim C <
    d  \rangle \, .
\]

\begin{thm}\label{thm:Vk_val}
    Let $C \subset \R^d$ be a cone and $H$ a hyperplane. Then
    \[
        V_k(C) \ = \ V_k(C \cap H^\le) + V_k(C \cap H^\ge)  - V_k(C \cap H) 
    \]
    as simple indicator functions.
\end{thm}
\begin{proof}
    Observe that $f = g$ for $f,g \in \sconeval_d$ if $f(x) = g(x)$ for almost
    all $x \in \R^d$.  Thus, let $x \in \R^d$ a \emph{generic} point.
    Let $C^\le, C^\ge,$ and $C^=$ denote $C \cap H^\le, C \cap H^\ge,$ and $C
    \cap H$, respectively.  Let $\pi_C(x) = y$ and $F \subseteq C$ the unique
    face with $y \in \relint(F)$. We may assume that $y \in H^\le$, that is,
    $y = \pi_{C^\le}(x)$ and $F^\le = F \cap H^\le$. Define $y^=, y^\ge$ with
    corresponding faces $F^=,F^\ge$. 

    If $y = y^\ge$, then $y \in H$ and $F \cap H \neq \emptyset$.  If $F
    \subseteq H$, then $F = F^\le = F^= = F^\ge$ and we are done. The case
    $F \not\subseteq H$ is not relevant, as $x$ is generic: perturbing $x$
    parallel to $\lin(F)$ moves $y$ away from $H$.

    If $y^\ge \neq y$, then $y^\ge = y^=$. Indeed, any point on
    the segment $z \in [y,y^\ge)$ satisfies $\|x-z\| < \|x-y^\ge\|$ and
    $[y,y^\ge]$ meets $H$. It follows that $F^\ge = F^=$ and $\dim F = \dim
    F^\le$. This proves the claim.
\end{proof}

The spherical volume $\sigma_{d-1}(C) \defeq \frac{\vol_d(C \cap
B_d)}{\vol_d(B_d)}$ is a simple valuation and hence extends to a linear
function $\sigma_{d-1} : \sconeval_d \to \R$.

\begin{cor}\label{cor:vk_val}
    The spherical volumes $v_k(C) = \sigma_{d-1}(V_K(C))$ are valuations.
\end{cor}

Similar to the situation for polytopes, the valuations
$V_k$ separate $\coneval_d$. The proof of the next result needs the
following observation.

\begin{lem}\label{lem:oV_polar}
    Let $C \subseteq \R^d$ be a cone and $0 \le k \le d$. Then $V_k(C) =
    V_{d-k}(C^\vee)$.
\end{lem}

\begin{thm}\label{thm:oV-injec}
    Let $C \subseteq \R^d$ be a cone and $\dim \lineal(C) \le k \le \dim C$.
    If $2k \neq d$, then $C$ can be recovered from
    $V_k(C)$.
\end{thm}
By the previous lemma, we can directly see, that for $2k = d$ we have
$V_k(C) = V_k(C^\vee)$, making this assumption necessary.
\begin{proof}
    First note that $V_k(C) = 0$ whenever $k > \dim C$ or, by
    Lemma~\ref{lem:oV_polar}, $k < \dim \lineal(C)$.  

    Now, let $S \subset \R^d$ be the collection of points for which $V_k(C)$
    does not vanish on a small neighborhood. This is the union of the
    interiors of $\Pi_F(C) = F + N_F C$, where $F$ ranges over the
    $k$-dimensional faces of $C$. It follows from~\eqref{eqn:PiF_intersect}
    that $\dim \Pi_F(C) \cap \Pi_G(C) < d$ for any two distinct $F, G \in
    \Faces_k(C)$ and hence we can recover the cones $\Pi_F(C)$ from $S$.

    We may use Lemma~\ref{lem:oV_polar} to assume that
    $k < \frac{d}{2}$.  Every $k$-face $E$ of $\Pi_F(C) = F + N_FC$ is
    of the form $E = E' + N_{E''}C$, where
    $E' \subseteq F \subseteq E''$ are faces with
    $d - k = \dim E'' - \dim E'$. We say that $E$ is \Defn{free}, if
    there exists no $k$-face $G$ of $C$, such that
    $E = \Pi_F(C) \cap \Pi_G(C)$. Equivalently, $E$ is not free, if
    and only if there exist another $k$-faces $G \neq F$ such that
    $E' \subseteq G \subseteq E''$. Thus $E$ is free, if and only if
    $F = E'$ or $F = E''$, and since $\dim E'' - \dim E' = d - k > k$,
    the case $F = E''$ is impossible. Note that if $F = E'$, then
    $E'' = C$, so $E = F + N_CC = F$.
    Therefore, the set of all free faces is precisely the set
    $k$-faces of $C$, from which we can recover $C$.
\end{proof}

\newcommand{\oX}{\overline{\mathcal{X}}}%
\subsection{Characteristic indicators}
Corollary~\ref{cor:klivans_swartz} can be generalized to the setting of
indicator functions. Let $\rho : \coneval_d \to \sconeval_d$ be the canonical
projection. Then we define $\oX : \Fans_d \to \sconeval_d[t]$ as
\[
    \oX_\Fan(t) \ \defeq \ \sum_{k = 0}^d \rho(V_k(\Fan)) t^k \ \in \
    \sconeval_d[t]\,.
\]
This is a natural generalization of $\ochi_\Fan(t)$.

\begin{cor}\label{cor:klivans_swartz_indi}
    Let $\Arr$ be an arrangement and $0 \leq k \leq d$. Then as elements of
    $\sconeval_d[t]$
    \begin{equation}\label{eq:indicator_klivans_swartz}
        \oX_{\Fan(\Arr)}(t) \ = \ \ochi_\Arr(t) \cdot \rho([\R^d]) \, .
    \end{equation}
\end{cor}
\begin{proof}
    The proof of Corollary~\ref{cor:klivans_swartz} applies verbatim on noting
    that 
    \[
        \oX_{\Fan(\Arr^k)}(t) \ = \ (t^k + t^{k-1}) \cdot
        \rho([\R^d]) \ = \ \ochi_{\Arr^k}(t) \cdot \rho([\R^d])
    \]
    for all $k$-singletons $\Arr^k$, $1 \leq k \leq d$.
\end{proof}

Already in dimension $2$ one can see that \eqref{eq:indicator_klivans_swartz}
holds only for generic points.  It was shown in~\cite{Kabluchko2020} that
the \emph{exceptional set} of points $x \in \R^d$ where $V_k(\Arr)(x) \neq
w_k(\Arr)$ is a hyperplane arrangement. We will slightly generalize the
results of~\cite{Kabluchko2020} with a simpler proof that prompts a simple
interpretation for the exceptional set.

For the proof, as well as the precise statement of our results, we
recall two well-known ring structures on $\coneval_d$. First, note
that $(\coneval_d, +, \cdot)$ is a commutative ring with unit $1 =
[\R^d]$ with respect to pointwise multiplication of functions.
 For $C, C' \in \Cones_d$ we have 
\[
    ([C] \cdot [C'])(x) \ \defeq \ [C](x) \cdot [C'](x) \ = \ [C \cap C'](x)\,.
\]
A second ring structure $(\coneval_d, +, \ast)$ is obtained with
respect to taking conical hulls $C \vee C' \defeq \cone(C \cup C')$:
\[
    ([C] \ast [C'])(x) \ \defeq \  [C \vee C'](x)\,.
\]
These two ring structures are related via polarity: The polarity map
given by $[C] \mapsto [C^\vee]$ is linear and since $(C \cap D)^\vee =
C^\vee + D^\vee$ for all $C, D \in \Cones_d$ for which $C \cup D$ is
convex, we see that polarity gives an isomorphism of rings $(\coneval, +,
\cdot) \cong (\coneval, +, \ast)$.

Let $\Arr$ be a hyperplane arrangement in a subspace $U$ of $\R^d$.
The \Defn{lattice of flats} $\Flats(\Arr)$ is the collection of
subspaces formed by intersections of hyperplanes in $\Arr$ ordered by
\emph{reverse} inclusion. The minimal element is $\Bot = U$ and
$\Top = \lineal(\Arr)$ is the
maximal element. The \Defn{M\"obius function} $\mu_\poset$ of a finite
partially ordered set $(\poset, \preceq)$ is recursively defined for
$x \preceq y \in \poset$ as follows: If $x = y$, then
$\mu_\poset(x, y) = 1$, otherwise
\[
    \mu_\poset(x, y) \ = \ -\sum_{x \preceq z \precneq y}
    \mu_\poset(x, z)\,.
\]
We will suppress the subscript of $\mu$ if $\poset$ is clear from the
context. It is well known that $\mu = \mu_{\Flats(\Arr)}$ alternates
in sign, or, more precisely,
$|\mu(L, M)| = (-1)^{\dim L - \dim M} \mu(L, M)$. Let
$\delta(L, K) = 1$ if $L = K$ and $= 0$ otherwise. Denote by
$\Faces(\Fan)$ the set of all nonempty faces of all cones of
$\Fan$. We will first show the following lemma, which will be
essential to our generalization of
Corollary~\ref{cor:klivans_swartz_indi}:

\begin{lem}\label{lem:key}
    Let $\Arr$ be a hyperplane arrangement with lattice of flats
    $\Flats(\Arr)$.  Then 
    \[
        \sum_{C \in \Fan(\Arr)} [C] \ = \ \sum_{L \in \Flats(\Arr)} |\mu(\Bot, L)|
        \cdot [L]
    \]
    as elements in $\coneval_d$.

\end{lem}

\newcommand\Euler{\mathcal{E}}%
Observe that the map $C \mapsto (-1)^{\dim C}[\relint(C)] \in \coneval_d$ is a
valuation. The \Def{Euler map} is the induced homomorphism $\Euler :
\coneval_d \to \coneval_d$ and one checks $\Euler \circ \Euler =
\mathrm{id}$. 

\begin{proof} 
    Let $\Arr$ be an arrangement in the $d$-dimensional subspace $U$ and let
    $L \in \Flats(\Arr)$. Since 
    \[
        [L] \ = \ \sum_{F \in \Faces(\Arr) , F
    \subseteq L } [\relint(F)] 
    \]
    It follows that 
    \[
        [L] \ = \ (-1)^{\dim L } \Euler([L]) \ = \ 
        \sum_{F \in \Faces(\Arr),  F \subseteq L} (-1)^{\dim L -
        \dim F} [F] \, .
    \]
    We calculate:
    \begin{align*}
      \sum_{L \in \Flats(\Arr)} |\mu(\Bot, L)| \cdot [L]
      \ =& \ \sum_{L \in \Flats(\Arr)} (-1)^{d - \dim L}\mu(\Bot, L) \sum_{F \in\Faces(\Arr) \atop F \subseteq L} (-1)^{\dim L - \dim F} [F]\\
      \ =& \ \sum_{F \in\Faces(\Arr)} (-1)^{d - \dim F} \sum_{L \in
        \Flats(\Arr) \atop F \subseteq L} \mu(\Bot, L) [F]\\
      \ =& \ \sum_{F \in\Faces(\Arr)} (-1)^{d - \dim F} \delta(\Bot, \lin(F)) [F]
           \ = \ \sum_{C \in \Fan(\Arr)} [C]\,.\qedhere
    \end{align*}
\end{proof}

\begin{thm}\label{thm:Vk_arr}
    For all $0 \leq k \leq d$, as elements in $\coneval_d$:
    \begin{equation}\label{eq:Vk_arr}
        V_k(\Arr)
        \ = \ \sum_{L \in \Flats_k(\Arr)}
        \Big(\sum_{K \in \Flats(\Arr) \atop K \subseteq L} |\mu(L, K)| \cdot [K]\Big) \ast
        \Big(\sum_{M \in \Flats(\Arr) \atop L \subseteq M} |\mu(\Bot, M)| \cdot [M^\perp]\Big)
    \end{equation}
\end{thm}

The proof is inspired by the arguments leading
to~\cite[Theorem~5]{KlivansSwartz}. The additional bookkeeping is delegated to
the ring structure on $\coneval_d$.

\begin{proof} We can rewrite the left-hand side:
    \begin{align*}
      V_k(\Arr) 
      \ &= \ \sum_{P \in \Fan(\Arr)} \sum_{F \in\Faces_k(P)} [F] \ast [N_FP]
      \ = \ \sum_{F \in\Faces_k(\Arr)} [F] \ast \sum_{P \in \Fan(\Arr) \atop F \leq P} [N_FP]\,.
    \end{align*}
    Let $L = \lin(F)$ and denote by $\Arr_L \defeq \{H \in \Arr : L \subseteq
    H\}$ the \Defn{localization} of $\Arr$ at $L$. Note that there is a
    one-to-one correspondence between the regions $C$ of $\Arr_L$ and $P \in
    \Fan(\Arr)$ with $F \subseteq P$ and that under this correspondence
    $C^\vee = N_FP$. Thus
    \begin{align*}
      \sum_{F \in\Faces_k(\Arr)} [F] \ast \sum_{P \in \Fan(\Arr) \atop F \leq P} [N_FP]
      \ = \ \sum_{L \in \Flats_k(\Arr)} \Big(\sum_{F \in \Fan(\Arr^L)} [F]\Big) \ast
      \Big(\sum_{C \in \Fan(\Arr_L)} [C^\vee]\Big)\,.
    \end{align*}
    Using Lemma~\ref{lem:key} on the inner sums gives the result:
    \begin{align*}
      \sum_{F \in \Fan(\Arr^L)} [F] \ = \ \sum_{K \in \Arr^L} |\mu_{\Arr^L}(L, K)| \cdot [K] \ = \ \sum_{K \in \Flats(\Arr) \atop K \subseteq L} |\mu_{\Arr}(L, K)| \cdot [K]\,,
    \end{align*}
    and:
    \begin{align}\label{eq:normal_cone}
      \sum_{C \in \Fan(\Arr_L)} [C^\vee] \ &= \ \Big(\sum_{C \in \Fan(\Arr_L)} [C]\Big)^\vee \ = \ \Big(\sum_{M \in \Arr_L} |\mu_{\Arr_L}(\Bot, M)| \cdot [M]\Big)^\vee\\
      \ &= \ \sum_{M \in \Flats(\Arr) \atop L \subseteq M} |\mu_{\Arr}(\Bot, M)| \cdot [M^\perp]\,.\qedhere
    \end{align}
\end{proof}

For direct comparison with Theorem 1.4 of \cite{Kabluchko2020}, note that for
$K \subseteq L \subseteq M \in \Flats(\Arr)$ we have $\dim(K + M^\perp) = d$
if and only if $K = L = M$. Thus for a generic point $x \in \R^d$ the
evaluation $([K] \ast [M^\perp])(x) = 0$ if $K \neq L$ or $L \neq M$. This
proves Corollary~\ref{cor:klivans_swartz_indi} coefficientwise:
\begin{align*}
    V_k(\Arr)(x) 
  \ = \ \sum_{L \in \Flats_k(\Arr)} |\mu(L, L)| \cdot |\mu(\R^d, L)| \cdot ([L] \ast [L^\perp])(x)
    \ = \ w_k(\Arr)\,,
\end{align*}

where we used that the terms $|\mu(\R^d, L)|$ sum up to $w_k$
(cf.~\cite[Section 3.10]{EC1}). Since all summands in
\eqref{eq:Vk_arr} are positive, this also shows that
${V_k(\Arr)(x) \geq w_k}$ for (nongeneric) points $x \in \R^d$. We
have $V_k(\Arr)(x) > w_k$ precisely when $x \in K + M^\perp$ for
$K \subseteq L \subseteq M$ such that $\dim L = k$ and
$\dim(K + M^\perp) = d - 1$. This arrangement has a simple
interpretation. Assume that $\Arr'$ is an arrangement of hyperplanes
in some subspace $U \subset \R^d$, then
$\Arr' + U^\perp = \{ H + U^\perp : H\in \Arr' \}$ is the
corresponding arrangement in $\R^d$ with lineality space $U^\perp$.

\begin{cor}
    Let $\Arr$ be an arrangement of hyperplanes in $\R^d$. Then
    $V_k(\Arr)(x) > w_k$ for some $0 \leq k \leq d$ if
    and only if $x$ is contained in the arrangement
    \[
        \Pi(\Arr) \ \defeq \ \bigcup_{L \in \Flats(\Arr)} (\Arr \Restrict L) +
        L^\perp \, .
    \]
\end{cor}

A point $x \in \R^d$ is contained in $(\Arr \Restrict L) + L^\perp$ if
and only if the orthogonal projection of $x$ onto $L$ is contained in
$\Arr \Restrict L$. The exceptional set of
Theorem~\ref{thm:Vk_arr} was also considered by
Lofano--Paolini~\cite{LofanoPaolini}. For a subspace $U \subset \R^d$ and
$x \in \R^d$, let $d_U(x) \defeq \| x - \pi_U(x)\|$ be the distance of
$x$ to $U$. Lofano--Paolini call a point $x \in \R^d$ \emph{generic}
with respect to an arrangement $\Arr$ if $d_L(x) \neq d_{L'}(x)$ for
all distinct $L,L' \in \Flats(\Arr)$. The set of non-generic
points lies on a collection of quadrics and is in general not a
hyperplane arrangement. A characterization of the points away from
$\Pi(\Arr)$ is as follows; cf.~Lemma 5.2 of~\cite{LofanoPaolini}.

\begin{lem}
    Let $\Arr$ be a hyperplane arrangement in $\R^d$ and $x \in \R^d$. Then $x
    \not \in \bigcup \Pi(\Arr)$ if and only if $d_{L'}(x) > d_L(x)$
    for all flats $L,L' \in \Flats(\Arr)$ with $L' \subsetneq L$.
    Equivalently, this is the case if and only if $d_H(x) > d_L(x)$ for all
    flats $L$ and $H \in \Arr \Restrict L$.
\end{lem}

\newcommand{\val}{\mathfrak{X}}%
\begin{rem}
    The proof of Theorem~\ref{thm:Vk_arr} and Lemma~\ref{lem:key}
    applies ad verbatim to affine hyperplane arrangements. For those,
    one has to work in the ring $(\val_d, +, \ast)$ of indicator
    functions of polyhedra instead of polyhedral cones, where $\ast$
    is defined via the Minkowski-sum, $[Q] \ast [Q'] = [Q + Q']$. The
    only problem is our use of polarity in
    \eqref{eq:normal_cone}. Here we use that \eqref{eq:normal_cone}
    holds in $(\coneval_d, +, \ast)$, which canonically embeds into
    $(\val_d, +, \ast)$, since $C \vee C' = C + C'$ for
    $C, C' \in \Cones_d$.
\end{rem}

\newcommand{\euler}{\eps}%
Finally, we also want to give a proof that $\ochi_\Fan(-1) = 0$, which
we also do on the level of indicator functions. For this, we give a
simple (algebraic) proof of a result of Schneider~\cite{Schneider2018}
and it's generalization to polyhedra by
Hug--Kabluchko~\cite{hug2016inclusion}. Recall that $\Faces(Q)$ denotes
the set of nonempty faces of a polyhedron $Q$. The
\Defn{Euler-Characteristic} of $Q$, is defined as
\[
    \euler(Q) \defeq \sum_{F \in \Faces(Q)} (-1)^{\dim F}\,.
\]
Recall that for a polyhedron $Q$ with $L = \lineal(Q)$ we have $\euler(Q) =
0$ if $L = 0$ and $Q$ is unbounded and $\euler(Q) = (-1)^{\dim
L}\euler(\pi_{L^\perp}(Q))$ otherwise.

\begin{thm}\label{thm:F-N_FC}
    Let $Q \subseteq \R^d$ be a polyhedron. Then, as elements in
    $\val_d$:
    \[
        \euler(Q) \cdot [\R^d] \ = \ \sum_{F \in \Faces(Q)} =
        (-1)^{\dim F} [F - N_FQ]\,.
    \]
\end{thm}
\begin{proof}
    Note that $(-1)^{\dim Q - \dim F} = \mu(F, Q)$, where
    $\mu = \mu_{\Faces(Q)}$. Denote by $T_FC \defeq (N_FC)^\vee$ the
    tangent cone of $C$ at $Q$.  By the Sommerville-relation
    (cf.~\cite[Lem.~4.1]{AS15}) we have for any cone $D \in \Cones_d$:
    \[
        \sum_{G \in \Faces(D)} (-1)^{\dim G} \cdot [T_GD] \ = \ (-1)^{\dim D}
        [-\relint D]\,.
    \]
    Applying $\Euler$ on both sides and using $\dim T_GD = \dim D$ we
    get
    \[
        \sum_{G \in \Faces(D)} (-1)^{\dim G} \cdot [\relint T_GD] \ = \
        (-1)^{\dim D} [-D]\,.
    \]
    If we set $D \defeq N_FQ$, we have $T_{G'}(N_FQ) = N_FG$ for
    $G' \defeq N_GQ$, so this
    reads:
    \begin{align*}
        (-1)^{\dim Q - \dim F} [-N_FQ]
        \ &= \ \sum_{G' \in \Faces(N_FQ)} (-1)^{\dim G'} \cdot [\relint T_{G'}(N_FQ)]\\
        \ &= \ \sum_{F \subseteq G \in \Faces(Q)} (-1)^{\dim Q - \dim G} \cdot [\relint N_FG]\,.
    \end{align*}
    
    Now, we can simply compute:
    \begin{align*}
      \sum_{F \in \Faces(Q)} (-1)^{\dim F} \cdot [F] \ast [-N_FQ]
      \ &= \ \sum_{F \in \Faces(Q)} [F] \ast \sum_{F \subseteq G \in \Faces(Q)} (-1)^{\dim G} \cdot [\relint N_FG]\\
      \ &= \ \sum_{G \in \Faces(Q)} (-1)^{\dim G} \sum_{F \in \Faces(G)} [F] \ast [\relint N_FG]\\
      \ &= \ \sum_{G \in \Faces(Q)} (-1)^{\dim G} \cdot [\R^d] \ = \ \euler(Q) \cdot [\R^d]\,. \qedhere
    \end{align*}
\end{proof}

That $\ochi_\Fan(-1) = 0$ follows now immediately by noting that
\[
    \sigma_{d-1}(F + N_FC) \ = \ \sigma_{d-1}(F - N_FC)
\]
for all cones $C \in \Fan$ and faces $F$ of $C$.

\bibliographystyle{siam}%
\bibliography{bibliography.bib}

\end{document}